\documentclass[12pt]{amsart}
  \RequirePackage{amsmath}
\RequirePackage{amssymb} \RequirePackage{amsthm}
\RequirePackage{epsfig} 
       \def\b{\beta}        \def\g{\gamma}
            \def\om{\omega} 
       \def\ti{\theta}       
                  \def\z{\zeta}
         
\def\a{\alpha}

\def\z{\zeta}

\DeclareMathOperator{\og}{O}
\newcommand{\DD}{\widehat{\mathcal{D}}}
\newcommand{\vg}{\widehat{v}}
\newcommand{\Vg}{\widehat{V}}
\newcommand{\fg}{\widehat{f}}
\newcommand{\go}{\widehat{g}}
\newcommand{\h}{\mathcal{H}}
\newcommand{\hti}{\widetilde{\mathcal{H}}}
\newcommand{\hg}{\mathcal{H}_{g}}
\def\D{{\mathbb D}}  
  \def\N{{\mathbb N}}

\def\B{{\mathcal B}}

\def\Dv{{\mathcal D_{v}}}

\def\({\left(}       \def\){\right)}

\newcounter{capa}

\newtheorem{theorem}{Theorem}
\newtheorem{lemma}[theorem]{Lemma}
\newtheorem{proposition}[theorem]{Proposition}

\newtheorem{corollary}[theorem]{Corollary}

\newtheorem{lettertheorem}{Theorem}

\theoremstyle{definition}

\theoremstyle{remark}

\theoremstyle{remarks}

\numberwithin{equation}{section}
\begin{document}
\title
[Schatten classes  ] {Schatten classes of generalized Hilbert operators }

\keywords{Dirichlet type  spaces, Schatten classes, generalized
Hilbert operators, Muckenhoupt weight, doubling weights }

\author[J. A. Pel\'aez]{Jos\'e \'Angel Pel\'aez}
\address{Jos\'e \'Angel Pel\'aez\\ Departamento de An\'alisis
Matem\'atico\\
Universidad de M\'alaga\\ Campus de Teatinos\\ 29071 M\'alaga\\
Spain}\email{japelaez@uma.es}

\author[D. Seco]{Daniel Seco}
\address{Departament de Matem\`atica Aplicada i An\`alisi, Facultat de Matem\`atiques, Universitat de Barcelona, Gran Via 585, 08007 Barcelona, Spain.}
 \email{dseco@mat.uab.cat}
\thanks{
The research leading to these results has received partial funding
from the Ram\'on y Cajal program of MICINN (Spain); from Ministerio
de Econom\'ia y Competitividad, Spain, projects MTM2011-24606,
MTM2011-25502 and MTM2014-52865-P; from La Junta de Andaluc{\'i}a,
(FQM210) and (P09-FQM-4468); and from the European Research Council
under the European Union's Seventh Framework Programme
(FP/2007-2013) / ERC Grant Agreement n.291497.}

\date{\today}

\keywords{Schatten class, Hilbert operator, Dirichlet spaces.}
\subjclass[2010]{Primary: 47B10 Secondary: 31C25, 47G10.}
\begin{abstract}
Let $\mathcal{D}_v$ denote the Dirichlet type space in the unit disc
induced by a radial weight $v$  for which $\vg(r)=\int_r^1 v(s)\,ds$
satisfies the doubling property $\int_r^1 v(s)\,ds\le C
\int_{\frac{1+r}{2}}^1 v(s)\,ds.$ In this paper, we characterize the
Schatten classes $S_p(\mathcal{D}_v)$ of the generalized Hilbert
operators
    \begin{equation*}
    \mathcal{H}_g(f)(z)=\int_0^1f(t)g'(tz)\,dt
    \end{equation*}
acting on $\mathcal{D}_v$, where $v$ satisfies the Muckenhoupt-type conditions
$$
\sup_{0<r<1}\left(\int_r^1 \frac{\widehat{v}(s)}{(1-s)^2} \,ds\right)^{1/2}
\left(\int_0^r \frac{1}{\widehat{v}(s)} \,ds\right)^{1/2}<\infty
$$
and
$$\sup_{0<
    r<1}\left(\int_{0}^r \frac{\widehat{v}(s)}{(1-s)^4}\,ds\right)^{\frac{1}{2}}
    \left(\int_{r}^1\frac{(1-s)^2}{\widehat{v}(s)}\,ds\right)^\frac{1}{2}<\infty.
   $$
For $p\ge 1$, it is proved that  $\hg\in S_p(\mathcal{D}_v)$ if and only if
    \begin{equation*}
    \int_0^1 \left((1-r)\int_{-\pi}^\pi |g'(re^{i\theta})|^2\,d\theta\right)^{\frac{p}{2}}\frac{dr}{1-r}
   <\infty.
    \end{equation*}
\end{abstract}
\maketitle
\section{Introduction and main results}
\par Let $\D$ denote the open unit disk of the complex
plane, and let $H(\D)$ be the class
of all analytic functions on $\D$.
 A function $v:\D\to
(0,\infty)$, integrable over $\D$, is called a weight. It is radial if $v(z)=v(|z|)$ for all $z\in\D$.
 The weighted Dirichlet
space
$\Dv$ consists of  $f\in H(\D)$ for
which
    $$
    \|f\|_{\Dv}^2=|f(0)|^2+\int_\D|f'(z)|^2 v(z)\,dA(z)<\infty,
    $$
where $dA(z)=\frac{dx\,dy}{\pi}$ is the normalized Lebesgue area
measure on $\D$. In this work, we will consider Dirichlet type
spaces $\Dv$ induced by weights in the class~$\DD$ of the radial
weights~$v$ for which $\vg(r)=\int_r^1 v(s)\,ds$ satisfy
 $sup_{0<r<1}\frac{\vg(r)}{\vg(\frac{1+r}{2})}<\infty.$
 The standard radial weights $v(z)=(1-|z|)^\a$,\,$\a>-1$ meet this doubling property. We  write $\mathcal{D}_\a$ for the Dirichlet type space induced by the standard weight $(1-|z|)^\a$.
The Hardy space $H^2$ consists of $f\in H(\D)$ for which
    $
    \|f\|_{H^2}=\lim_{r\to1^-}M_2(r,f)<\infty,$\,
where
    $$
    M_2(r,f)=\left (\frac{1}{2\pi }\int_{-\pi}^{\pi}
    |f(re^{i\theta})|^2\,d\theta\right )^{\frac{1}{2}}.
    $$
The classical Littlewood-Paley formula says that
$H^2=\mathcal{D}_1$. We refer the reader to \cite{Duren1970} for
background information on this space. We denote by $A^2_\om$ the
Bergman space induced by a weight $\om$ (see
\cite[Chapter~$1$]{PelRat}). Moreover,  if $\om$ is radial then
  $A^2_\om=\mathcal{D}_{\om^\star}$, where
    $$
    \omega^\star(z)=\int_{|z|}^1 s \log\frac{s}{|z|}\omega(s)\,ds,\quad z\in\D\setminus\{0\}.
    $$
See \cite[Theorem $4.2$]{PelRat} for the details.

\par Every $g\in H(\D)$ induces an operator, that we call \emph{the generalized Hilbert
operator}
$\hg$, defined by
    \begin{equation}\label{H-g}
    \mathcal{H}_g(f)(z)=\int_0^1f(t)g'(tz)\,dt, \quad f\in H(\D).
    \end{equation}

     The sharp condition
 \begin{equation}\label{eq:vg2}
\int_{0}^1\frac{(1-s)^2}{\vg(s)}\,ds<\infty
\end{equation}
 ensures that the integral in \eqref{H-g} defines an analytic function for each $f\in
 \Dv$ (see Lemma~\ref{le:welldef} below).

 The choice
$g(z)=\log\frac{1}{1-z}$ in \eqref{H-g} gives an integral
representation of the classical Hilbert operator $\h$. The Hilbert
operator $\h$ is a model of Hankel operator, and has been the object
of previous studies such as \cite{AlMonSa,Di,DJV}, where the authors
dealt with questions related to the boundedness, the operator norm
and the spectrum of $\h$. This has revealed a natural connection
between $\h$ and other classical objects: the weighted composition
operators, the Szeg\"{o} projection and the Legendre functions of
the first kind. The Hilbert operator is bounded on the classical
Dirichlet type space $\mathcal{D}_\a$ if and only if $\a\in (0,2)$,
as was shown in \cite{D2,GaGiPeSis}. In fact, if $\a\ge 2$ there is
$f\in \mathcal{D}_\a$, $f\ge 0$ on $[0,1)$ such that $\int_0^1
f(t)\,dt=\infty$.

The generalized Hilbert operator $\hg$ was introduced recently in
\cite{GaGiPeSis}, where it is provided, among other results, a
description of the $g\in H(\D)$ such that $\hg$ is bounded, compact
or Hilbert-Schmidt on $\mathcal{D}_\a$,\,$\a\in (0,2)$. In
\cite{PelRathg}, the authors solve the question of when is $\hg$
bounded or compact between weighted Bergman spaces $A^p_\om$ and
$A^q_\om$,\, $1<p,q<\infty$, induced by a large class of radial
weights.
\medskip

The primary aim of this paper is to determine  the membership in
 Schatten ideals $\mathcal{S}_p(\Dv)$ of generalized Hilbert operators $\hg$
acting on Dirichlet type spaces $\Dv$, $v\in\DD$.  This leads us to consider the following
 spaces.
For $0<p<\infty$, the mixed norm space $\B(2,p)$ consists of $g\in
H(\D)$ such that
    \begin{equation*}
    \left\|g\right\|^q_{\B(2, p)}= |g(0)|^p+\int_0^1 M^p_2(r,g')(1-r)^{\frac{p}{2}-1}\,dr<\infty.
    \end{equation*}
Let us observe that $\B(2,2)$ is nothing else but the classical Dirichlet space $\mathcal{D}_0=\mathcal{D}$.
The space
$\B(2,\infty)$ consists of
$g\in H(\D)$ such that
    \begin{equation*}
    \left\|g\right\|_{\B(2,\infty)}=|g(0)|+\sup_{0<r<1} M_2(r,g')(1-r)^{\frac{1}{2}}<\infty.
    \end{equation*}

 A classical result of Hardy and Littlewood  \cite[Chapter~5]{Duren1970}
  asserts that   $\B(2, \infty)$ coincides with the mean Lipschitz space
$\Lambda\left(2,\frac{1}{2}\right)$  of the $g\in H(\D)$ having a
non-tangential limit $g(e^{i\theta})$ almost everywhere and such
that $$ \omega_2(g, t)=\og(t^\frac{1}{2}), \quad t\to 0, $$ where
$$
\omega_2(g, t)=\sup_{0<h\le t}\left(\int_0^{2\pi}
|g(e^{i(\theta+h)})-g(e^{i\theta})|^2 \frac{d\theta}{2\pi}\right)^{1/2}
$$
is the integral modulus of continuity of order $2$.

The corresponding \lq\lq little
oh\rq\rq \, mean Lipschitz space   $b (2, \infty)$, usually denoted by $\lambda\left(2,\frac12\right)$,  consists of
$g\in H(\D)$ such that
$$\lim_{r\to 1^-} M_2(r,g')(1-r)^{1/2}=0.$$

\par The next theorem is the main result of this paper.
\begin{theorem}\label{th:mainhg}
Let $g\in H(\D)$,  $1\le p\le \infty$ and $v\in\DD$ which satisfies the conditions
\begin{equation}\label{eq:minftyDv}
M_1(v)=\sup_{0<r<1}\left(\int_r^1 \frac{\widehat{v}(s)}{(1-s)^2} \,ds\right)^{1/2}
\left(\int_0^r \frac{1}{\widehat{v}(s)} \,ds\right)^{1/2}<\infty,
\end{equation}
and
\begin{equation}\label{eq:h1}
    \begin{split}
    M_2(v)&=\sup_{0<
    r<1}\left(\int_{0}^r \frac{\widehat{v}(s)}{(1-s)^4}\,ds\right)^{\frac{1}{2}}
    \left(\int_{r}^1\frac{(1-s)^2}{\widehat{v}(s)}\,ds\right)^\frac{1}{2}<\infty.
    \end{split}
    \end{equation}
Then, the following conditions are equivalent:
\begin{enumerate}
\item[\rm(i)] $\mathcal{H}_g \in S_p(\Dv)$;
\item [\rm(ii)] $g\in \B(2,p)$.
\end{enumerate}
Moreover, $$\| \hg\|_{ S_p(\Dv)}\asymp \| g-g(0)\|_{ \B(2,p)}.$$
\end{theorem}
It will also be proved (see Proposition~\ref{compactness} below)
that $\hg$ is compact on $\Dv$ if and only if $g\in b (2,\infty)$.
Notice that the conditions that characterize the Schatten classes do
not depend on the weight defining the space.

If $v(z)=(1-|z|)^\a$ is a standard weight, \eqref{eq:minftyDv} holds
if and only if $\a>0$, and \eqref{eq:h1} is satisfied if and only if
$\a<2$, so both conditions hold simultaneously if and only if $\a\in
(0,2)$. Therefore, in particular, Theorem~\ref{th:mainhg} provides a
characterization of Schatten classes of generalized Hilbert
operators $\hg$ acting on the Hardy space $H^2$ ($\a=1$) and
standard Bergman spaces $A^2_\beta$, $\beta\in (-1,0)$.
 Going further the next result  follows from Theorem~\ref{th:mainhg}.

\begin{corollary}\label{co:hgbergman}
Let $g\in H(\D)$,  $1\le p\le \infty$ and  $\om\in\DD$ which satisfies the condition
\begin{equation}\label{M2condition}
    \sup_{0<r<1}
    \bigg(\int_{0}^r\frac{\widehat{\om}(t)}
    {(1-t)^{2}}\,dt\bigg)^\frac{1}{2} \bigg(\int_{r}^1\frac{1}{\widehat{\om}(t)}\,dt\bigg)^\frac{1}{2}<\infty.
    \end{equation}
Then, the following conditions are equivalent:
\begin{enumerate}
\item[\rm(i)] $\mathcal{H}_g \in S_p(A^2_\om)$;
\item [\rm(ii)] $g\in \B(2,p)$.
\end{enumerate}
\end{corollary}
The Muckenhoupt-type condition \eqref{M2condition} arises in the
study of generalized Hilbert operators $\hg$ on weighted Bergman
spaces in~\cite{PelRathg}, where the authors describe the $g\in
H(\D)$ such that $\hg$ is bounded, compact or Hilbert-Schmidt on
$A^2_\om$ for the the subclass of $\DD$ consisting of regular
weights.

\par From now on, for each radial weight $v$ and $x\in\mathbb{R}$ let us denote $\Vg_{x}(r)=\frac{\vg(r)}{(1-r)^x}$.
Our approach to prove the case $p=\infty$ of Theorem~\ref{th:mainhg} reveals the  role of  Muckenhoupt type conditions
\eqref{eq:minftyDv} and \eqref{eq:h1}. On one hand,  \eqref{eq:minftyDv}   allows to prove that $L^2_{\Vg_2}$ is a   natural restriction of $\Dv$
to functions defined on $[0, 1)$. On the other hand,    the sublinear Hilbert
operator  defined by
    $$
    \hti(f)(z)=\int_0^1\frac{|f(t)|}{1-tz}\,dt$$
 behaves like a kind of maximal function for all generalized
Hilbert operators $\hg$ such that $g\in \B(2,\infty)$, and hence, it
will be essential to study its boundedness on $L^2_{\Vg_2}$.

\begin{theorem}\label{th:hilbertop}
Let  $v\in\DD$ which satisfies  the conditions \eqref{eq:vg2} and
\eqref{eq:minftyDv}. Then the following assertions are equivalent:
    \begin{enumerate}
    \item[\rm(i)] $\h:\,L^2_{\Vg_{2}}\to \Dv$ is bounded;
    \item[\rm(ii)] $\hti:\,L^2_{\Vg_{2}}\to \Dv$ is bounded;
    \item[\rm(iii)] $v$ satisfies the Muckenhoupt type condition \eqref{eq:h1}.
\end{enumerate}
Moreover,
 if $M_2(v)<\infty$, then
    $$
 \frac{M_2(v)}{M_1(v)}\lesssim   \|\h\|_{\left(L^2_{\Vg_{2}},\Dv\right)}\asymp
    \|\hti\|_{\left(L^2_{\Vg_{2}},\Dv\right)}\lesssim M_1(v)M_2(v).
    $$
\end{theorem}

We obtain as a byproduct the following result which extends several
results in the literature \cite{D2,PelRathg}.

\begin{corollary}\label{co:hilbertop}
Let  $v\in\DD$ which satisfies  the conditions \eqref{eq:minftyDv}
and \eqref{eq:h1}. Then, both the Hilbert operator~$\h$ and the
sublinear Hilbert operator $\hti$ are bounded on $\Dv$.
\end{corollary}

Throughout the paper   the letter $C=C(\cdot)$ will denote an
absolute constant whose value depends on the parameters indicated
in the parenthesis, and may change from one occurrence to another.
We will use the notation $a\lesssim b$ if there exists a constant
$C=C(\cdot)>0$ such that $a\le Cb$, and $a\gtrsim b$ is understood
in an analogous manner. In particular, if $a\lesssim b$ and
$a\gtrsim b$, then we will write $a\asymp b$.

\section{The Hilbert operator on $\Dv$}
\subsection{Some results on weights}
The following lemma provides useful characterizations of weights in
$\DD$. For a proof, see \cite{PelSum14}. Given a radial weight $v$,
we  write $v_x=\int_0^1 s^x v(s)\,ds$, $x>-1$.

\begin{lemma}\label{Lemma:weights-in-D-hat}
Let $\om$ be a radial weight. Then the following conditions are equivalent:
\begin{itemize}
\item[\rm(i)] $\om\in\DD$;
\item[\rm(ii)] There exist $C=C(\om)\ge 1$ and $\b=\b(\om)>0$ such that
    \begin{equation*}
    \begin{split}
    \widehat{\om}(r)\le C\left(\frac{1-r}{1-t}\right)^{\b}\widehat{\om}(t),\quad 0\le r\le t<1;
    \end{split}
    \end{equation*}
\item[\rm(iii)] There exist $C=C(\om)>0$ and $\gamma=\gamma(\om)>0$ such that
    \begin{equation*}
    \begin{split}
    \int_0^t\left(\frac{1-t}{1-s}\right)^\g\,\om(s)\,ds
    \le C\widehat{\om}(t),\quad 0\le t<1;
    \end{split}
    \end{equation*}
\item[\rm(iv)] There exist $C=C(\om)>0$ and $\eta=\eta(\om)>0$ such that
    \begin{equation*}
    \begin{split}
    \om_x\le C\left(\frac{y}{x}\right)^{\eta}\om_y,\quad 0<x\le y<\infty;
    \end{split}
    \end{equation*}
    \item[\rm(v)] There exists $C=C(\om)>0$ such that $\om_{n}\le C\om_{2n}$ for all $n\in\N$;

\item[\rm(vi)]
    \begin{equation*}
    \om_x \asymp\widehat{\om}\left(1-\frac1x\right),\quad
    x\in[1,\infty);
    \end{equation*}
\item[\rm(vii)] There exists $\lambda=\lambda(\om)\ge0$ such that
    $$
    \int_\D\frac{\om(z)\,dA(z)}{|1-\overline{\z}z|^{\lambda+1}}\asymp\frac{\widehat{\om}(\zeta)}{(1-|\z|)^\lambda},\quad \z\in\D;
    $$
    \item[\rm(viii)] $\om^\star(z)\asymp\widehat{\om}(z)(1-|z|)$ for $|z|\ge \frac12$.
\end{itemize}
\end{lemma}

\par The following technical lemma will be used in the proof of Theorem~\ref{th:hilbertop}.
\begin{lemma}\label{le:weight1}
Let $v$ be a radial weight.  If \eqref{eq:h1} holds, then
$$\sup_{0\le
    r<1}\left(\int_{0}^r \frac{v(s)}{(1-s)^3}\,ds\right)^{\frac{1}{2}}
    \left(\int_{r}^1\frac{1}{\Vg_2(s)}\,ds\right)^\frac{1}{2}<\infty.$$
\end{lemma}
\begin{proof}
For $0<r<1$, Fubini's theorem gives
\begin{equation*}
\int_0^r \frac{\vg(s)}{(1-s)^4}\,ds= \int_0^r
\frac{v(s)[1-(1-s)^3]}{3(1-s)^3}\,ds+
\vg(r)\frac{1-(1-r)^3}{3(1-r)^3}.
\end{equation*}

We may assume that $r\in[\frac12, 1)$, and then we have
\begin{equation*}\begin{split}
\int_{0}^r \frac{v(s)}{(1-s)^3}\,ds&\le C\int_0^{\frac12}v(s)\,ds+C\int_{\frac12}^r \frac{v(s)[1-(1-s)^3]}{(1-s)^3}\,ds
\\ & \le C\int_0^{\frac12}v(s)\,ds+C\int_0^r \frac{\vg(s)}{(1-s)^4}\,ds.
\end{split}\end{equation*}
By \eqref{eq:h1}, there is $C=C(v)$ such that
$\int_{0}^1\frac{1}{\Vg_2(s)}\,ds\le CM_2(v)<\infty$. Therefore, the
above inequality yields
\begin{equation*}\begin{split} &\sup_{\frac12\le
    r<1}\left(\int_{0}^r \frac{v(s)}{(1-s)^3}\,ds\right)^{\frac{1}{2}}
    \left(\int_{r}^1\frac{1}{\Vg_2(s)}\,ds\right)^\frac{1}{2}
    \\ &\le
     C
     \sup_{\frac12\le
    r<1}\left(\int_0^{\frac12}v(s)\,ds+\int_0^r \frac{\vg(s)}{(1-s)^4}\,ds\right)^{\frac{1}{2}} \left(\int_{r}^1\frac{1}{\Vg_2(s)}\,ds\right)^\frac{1}{2}
     \\ & \le C\left(\int_0^{\frac12}v(s)\,ds\int_{0}^1\frac{1}{\Vg_2(t)}\,dt<\infty\right)^{\frac{1}{2}}+
     CM_2(v)\le CM_2(v)<\infty.
     \end{split}\end{equation*}
This finishes the proof.
\end{proof}

\subsection{Hardy-Littlewood type inequalities}

The first result in this subsection gives a sharp condition that
ensures that $\hg$ is well defined on $\Dv$.

\begin{lemma}\label{le:welldef}
Let $v$ be a radial weight  which satisfies \eqref{eq:vg2}. Then there is a positive constant $C(v)$ such that
\begin{equation}\label{eq:welldef}
\int_0^1 |f(t)|\,dt\le C(v) \|f\|_{\Dv},\quad f\in H(\D).
\end{equation}
\end{lemma}
\begin{proof}
For any $f(z)=\sum_{k=0}^\infty \fg(k)z^k\in H(\D)$, it is clear
that
\begin{equation}\label{eq:welldef2}
\int_0^1 |f(t)|\,dt  \le |\fg(0)|+\left(\sum_{k=1}^\infty
k^2|\fg(k)|^2v_{2k-1}\right)^{1/2} \left(\sum_{k=1}^\infty
\frac{1}{k^2(k+1)^2v_{2k-1}}\right)^{1/2}.\end{equation} By
Lemma~\ref{Lemma:weights-in-D-hat}(v)(vi)
\begin{equation*}\begin{split}
\sum_{k=1}^\infty \frac{1}{k^2(k+1)^2v_{2k-1}} & \le C
\sum_{k=1}^\infty \frac{1}{\vg\left(
1-\frac{1}{k+1}\right)}\int_{1-\frac{1}{k+1}}^{1-\frac{1}{k+2}}(1-s)^2\,ds
\\ & \le C\int_{\frac12}^1 \frac{(1-s)^2}{\vg(s)}\,ds\le C \int_{0}^1 \frac{(1-s)^2}{\vg(s)}\,ds \le C.
\end{split}\end{equation*}
This together with \eqref{eq:welldef2} and the identity
 $
    \|f\|_{\Dv}^2=|\fg(0)|^2+2\sum_{k=1}^\infty |k\fg(k)|^2v_{2k-1},
    $
 implies \eqref{eq:welldef}.
\end{proof}

\par A classical result of Hardy-Littlewood (\cite[Theorem $5.11$]{Duren1970}) says that
$$\int_0^1 M^2_\infty(r,f)\,dr\le C  \|f\|^2_{H^2}.$$
See also the classical F\'ejer-Riesz inequality \cite[Theorem
$3.13$]{Duren1970}. Applying this inequality to dilated functions
$f_r(z)=f(rz)$,\, $0<r<1$, and integrating respect to a radial
weight $\om$, it can be easily obtained that
    \begin{equation}\label{eq:minftyA2}
    \int_0^1
    M^2_\infty(r,f)\,\widehat{\om}(r)\,dr\le C\|f\|^2_{A^2_\om}.
    \end{equation}
The next result shows  a Hardy-Littlewood type inequality in a setting of weighted Dirichlet  spaces.

\begin{lemma}\label{le:minftyDv}
Let $v$ be  a radial weight  which satisfies \eqref{eq:minftyDv}.
Then, there exists $C=C(v)>0$ such that
\begin{equation}\label{eq:minftyDv2}
\int_0^1 M_\infty^2(s,f) \frac{\widehat{v}(s)}{(1-s)^2} \,ds\le CM^2_1(v) \|f\|^2_{\Dv},\quad f\in H(\D).
\end{equation}
\end{lemma}
\begin{proof}
By condition \eqref{eq:minftyDv}  there is a constant $C=C(v)>0$
such that $\int_0^1 \frac{\widehat{v}(s)}{(1-s)^2} \,ds\le
CM^2_1(v)<\infty$. Using \cite[Theorem $1$]{Muckenhoupt1972} with
    $$
    U^{2}(s)=\left\{
        \begin{array}{cl}
        \frac{\widehat{v}(s)}{(1-s)^2}, &   0\le s<1\\
        0, & s\ge1
        \end{array}\right.
    $$
    and
    $$
    V^{2}(s)=\left\{
        \begin{array}{cl}
        \widehat{v}(s), &   0\le s<1\\
        0, & s\ge 1
        \end{array}\right..
    $$
we obtain
\begin{equation}\begin{split}\label{eq:minftyDv3}
&\int_0^1 M_\infty^2(s,f) \frac{\widehat{v}(s)}{(1-s)^2} \,ds
\\ & \le CM^2_1(v)|f(0)|^2+ \int_0^1 \left(\int_0^s M_\infty(r,f')\,dr\right)^2 \frac{\widehat{v}(s)}{(1-s)^2} \,ds
\\ & \lesssim CM^2_1(v) \left(|f(0)|^2+ \int_0^1 M_\infty^2(s,f')
\widehat{v}(s)\,ds\right).
\end{split}\end{equation}
Joining \eqref{eq:minftyA2} and \eqref{eq:minftyDv3}, we get
\eqref{eq:minftyDv2}.
\end{proof}
It is worth mentioning that the inequality $$M_\infty(r,f')\le C \frac{ M_\infty\left(\frac{1+r}{2},f\right)}{1-r},\quad 0<r<1,$$ implies
the reverse inequality of \eqref{eq:minftyDv3}
 for any $f\in H(\D)$ and $v\in\DD$.

\subsection{Proof of Theorem \ref{th:hilbertop}.}

\par It is clear that (ii)$\Rightarrow$(i).
\par (i)$\Rightarrow$(iii).
This part of the proof uses ideas from \cite{Muckenhoupt1972}. For
$r\in[0,1)$, set $\phi_r(t)=\frac{1}{\Vg_2(t)}\chi_{[r,1)}(t)$, so
that $\phi_r\in L^2_{\Vg_2}$  by~\eqref{eq:vg2}. Here, as usual,
$\chi_E$ stands for the characteristic function of a set $E$.
Bearing in mind Lemma~\ref{le:minftyDv}, we deduce
    \begin{equation*}
    \begin{split}
    \|\h(\phi_r)\|_{L^2_{\Vg_2}}\lesssim M_1(v) \|\h(\phi_r)\|_{\Dv} \le
   M_1(v) \|\h\|_{\left(L^2_{\Vg_2},\Dv\right)}\|\phi_r\|_{L^2_{\Vg_2}},
    \end{split}
    \end{equation*}
    and hence
    \begin{equation}\label{eq:j12}
    \int_{0}^1 \Vg_2(s)\left(\int_r^1\frac{1}{\Vg_2(t) (1-ts)}\,dt\right)^2\,ds \lesssim \int_r^1 \frac{1}{\Vg_2(t)}\,dt.
    \end{equation}
On the other hand,
    \begin{equation*}
    \begin{split}
     \int_{0}^r \Vg_2(s)\left(\int_r^1\frac{1}{\Vg_2(t) (1-ts)}\,dt\right)^2\,ds
    \ge \frac{1}{4}\left(\int_{0}^r \Vg_4(s)\,ds\right)\left(\int_r^1\frac{1}{\Vg_2(t)}\,dt\right)^2,
    \end{split}
    \end{equation*}
and this, together with \eqref{eq:j12}, implies
    $$M_2(v)\lesssim M_1(v) \|\h\|_{\left(L^2_{\Vg_2},\Dv\right)}<\infty.$$
\par (iii)$\Rightarrow$(ii).  For any $\phi\in L^2_{\Vg_2}$,
    $$
    \left(\hti(\phi)\right)'(z)=\int_0^1\frac{t|\phi(t)|}{(1-tz)^{2}}\,dt,
    $$
and so Minkowski's inequality in continuous form yields
    \begin{equation*}
    \begin{split}
    M_2(r,\left(\hti(\phi)\right)')
     &=\left(\frac{1}{2\pi}\int_0^{2\pi}\left|\int_0^1\frac{|\phi(t)|t}{(1-tre^{i\ti})^{2}}\,dt\right|^2\,d\ti\right)^\frac12
    \\
    &\le\int_0^1|\phi(t)|\left(\int_0^{2\pi}\frac{d\ti}{|1-tre^{i\ti}|^{4}}\right)^\frac12\,dt
    \asymp\int_0^1\frac{|\phi(t)|}{(1-tr)^{3/2}}\,dt.
    \end{split}
    \end{equation*}
Hence, decomposing the range of variation of $t$, we obtain
    \begin{equation}
    \begin{split}\label{eq:j13}
    \|\hti(\phi)\|^2_{\Dv}\lesssim I_1(r)+I_2(r)+|\hti(\phi)(0)|^2
    \end{split}
    \end{equation}
where
    $$
    I_1(r)=\int_0^1\left(\int_0^r\frac{|\phi(t)|}{(1-t)^{3/2}}\,dt\right)^2v(r)\,dr
    $$
and
    $$
    I_2(r)=\int_0^1\left(\int_r^1\frac{|\phi(t)|}{(1-tr)^{3/2}}\,dt\right)^2 v(r)\,dr.
    $$
\par By \eqref{eq:h1}
\begin{equation}\label{eq:hcero}
|\hti(\phi)(0)|^2\le \| \phi \|^2_{L^2_{\Vg_{2}}}\int_0^1
\frac{1}{\Vg_2(s)}\,ds\le C(v)M^2_2(v) \| \phi
\|^2_{L^2_{\Vg_{2}}}<\infty.
\end{equation}
The inequality
    \begin{equation}\label{eq:j13n}
    \begin{split}
    I_1(r)\lesssim \|\phi\|^2_{L^2_{\Vg_2}}
    \end{split}
    \end{equation}
can be written as
    $$
    \int_0^1\left(\int_0^r \Phi(t)\,dt\right)^2 U^2(r)\,dr\lesssim \int_0^1 \Phi^2(r)V^2(r)\,dr,
    $$
where
    $$
    U^{2}(x)=\left\{
        \begin{array}{cl}
        v(x), &   0\le x<1\\
        0, & x\ge1
        \end{array}\right.,
    $$
    $$
    V^{2}(x)=\left\{
        \begin{array}{cl}
        (1-x)\widehat{v}(x), &   0\le x<1\\
        0, & x\ge1
        \end{array}\right.,
    $$
and $\Phi(t)=\frac{|\phi(t)|}{(1-t)^{\frac32}}$. From this, by
\cite[Theorem~1]{Muckenhoupt1972},
 \eqref{eq:j13n} holds if and only if
\begin{equation*}
M_3(v)=\sup_{0<r<1}\vg(r)^{\frac12}\left(\int_{0}^r \frac{1}{(1-s)\vg(s)}\,ds \right)^{\frac12}<\infty.
\end{equation*}
Using  \eqref{eq:minftyDv}, we get
\begin{equation*}\begin{split}
\int_0^r \frac{1}{(1-s)\vg(s)}\,ds &\le\frac{1}{1-r} \int_0^r \frac{1}{\vg(s)}\,ds
 \le \frac{M^2_1(v)}{(1-r)\int_r^1 \Vg_2(s)\,ds}
\\ & \le \frac{M^2_1(v)}{(1-r)\int_r^{\frac{1+r}{2}} \Vg_2(s)\,ds}
 \le \frac{CM^2_1(v)}{\vg\left(\frac{1+r}{2}\right)}
 \le \frac{CM^2_1(v)}{\vg\left(r\right)},
\end{split}\end{equation*}
which implies that $ M_3(v)\le CM_1(v)$, and so by
\cite[Theorem~1]{Muckenhoupt1972}
 \begin{equation}
    \begin{split}\label{eq:j14}
    I_1(r)\le C(v)M^2_1(v)\|\phi\|^2_{L^2_{\Vg_2}}.
    \end{split}
    \end{equation}
Moreover,  by
applying \cite[Theorem~2]{Muckenhoupt1972} with
    $$
    U^{2}(x)=\left\{
        \begin{array}{cl}
        \frac{v(x)}{(1-x)^3}, &   0\le x<1\\
        0, & x\ge1
        \end{array}\right.,
    $$
and
    $$
    V^{2}(x)=\left\{
        \begin{array}{cl}
        \Vg_2(x), &   0\le x<1\\
        0, & x\ge1
        \end{array}\right.,
    $$
we deduce that
    \begin{equation*}
    \begin{split}
    I_2(r)\lesssim\int_0^1\left(\int_r^1\phi(t)\,dt\right)^2\frac{v(r)}{(1-r)^3}\,dr\le M^2_4(v)\|\phi\|^2_{L^2_{\Vg_2}},
    \end{split}
    \end{equation*}
    holds whenever
 $$M_4(v)=   \sup_{0\le
    r<1}\left(\int_{0}^r \frac{v(s)}{(1-s)^3}\,ds\right)^{\frac{1}{2}}
    \left(\int_{r}^1\frac{1}{\Vg_2(s)}\,ds\right)^\frac{1}{2}<\infty.$$
By Lemma~\ref{le:weight1}, $M_4(v)\le CM_2(v)<\infty$.
This together with \eqref{eq:j13}, \eqref{eq:hcero}  and \eqref{eq:j14} gives
(iii)$\Rightarrow$(ii). Going further, since there is an absolute constant $K>0$ such that $\min\{M_1(v),M_2(v)\}>K$, we get
   $$ \|\hti\|_{\left(L^2_{\Vg_{2}},\Dv\right)}\lesssim M_1(v)M_2(v).
    $$ \hfill$\Box$

Corollary~\ref{co:hilbertop} follows from Theorem~\ref{th:hilbertop}
and Lemma~\ref{le:minftyDv}.

\section{Proof of Theorem~\ref{th:mainhg}}
The right choice of the norm used is in many cases a key to a good
understanding of how a concrete operator acts in a given space. Here
the spaces $\B(2,p)$ will be equipped with an $l^p$- norm of the
$H^2$ norms of dyadic blocks of the Maclaurin series. In fact, a
calculation shows that
\begin{equation*}
    \left\|g\right\|^2_{\B(2, \infty)}\asymp |\widehat{g}(0)|^2+\sup_{n\in\N}\left(\frac{1}{2^n}\sum_{k\in I(n)}k^2|\widehat{g}(k)|^2\right),
    \end{equation*}
    where $g(z)=\sum_{k=0}^\infty \widehat{g}(k)z^k$ and $I(n)=\{k\in\N: 2^n\le k<2^{n+1}-1\}$,\, $n\in\N$.
     The same techniques allow us to prove that
\begin{equation*}
  g \in b (2, \infty) \Leftrightarrow \lim_{n\to\infty}\frac{1}{2^n}\sum_{k\in I(n)}k^2|\widehat{g}(k)|^2=0.
    \end{equation*}
Throughout this section, the expression $\widetilde{g_k}$ will
denote $k^2|\hat{g}(k)|^2$. Using \cite[Theorem~$1$]{MatelPav}
 it  can also be proved that
\begin{equation*}
    \left\|g\right\|^p_{\B(2, p)}\asymp |\widehat{g}(0)|^p+\sum_{n=0}^\infty 2^{-\frac{np}{2}}\left(\sum_{k\in I(n)}\widetilde{g_k}\right)^{\frac{p}{2}},
    \quad 0<p<\infty.
    \end{equation*}

\subsection{Boundedness and compactness}

\begin{proposition}\label{p=infty}
Let $g\in H(\D)$ and $v\in \DD$ which satisfies the conditions \eqref{eq:minftyDv} and \eqref{eq:h1}.   Then
$\mathcal{H}_g$ is bounded on $\Dv$ if and only $g\in \B(2,\infty)$. Moreover,
$$\| \hg\|\asymp \| g-g(0)\|_{ \B(2,\infty)}.$$
\end{proposition}

\begin{proof}

We use that the  norm of $\mathcal{H}_g(f)$ can be computed from the Taylor coefficients of $g$ and the moments of $f\chi_{[0,1)
}$
\begin{equation}\label{eqn10}
\|\mathcal{H}_g(f) \|^2_{\Dv} = \left|\go(1)\int_0^1
f(t)\,dt\right|^2+ \sum_{j=1}^{\infty}
j^2\widetilde{g_{j+1}}\left|\int_0^1 t^jf(t)\,dt\right|^2v_{2j-1}.
\end{equation}
\par Assume that $g\in \B(2,\infty)$, and let us see first that
\begin{equation}\label{eqn11}\sum_{j=2}^{\infty} j^2\widetilde{g_{j+1}} \left|\int_0^1
t^jf(t)\,dt\right|^2v_{2j-1} \le C \| g-g(0)\|^2_{ \B(2,\infty)}\|
f\|^2_{\Dv}. \end{equation} Indeed, the left-hand side of the above
can be decomposed in dyadic pieces in terms of the parameter $j$,
and is therefore dominated by
\begin{equation*}\begin{split}
&C \sum_{n=1}^\infty v_{2^{n+1}-1} 2^{2n} \left(\int_0^1
t^{2^n}|f(t)|\,dt\right)^2\sum_{j\in I(n)}\widetilde{g_{j+1}}
\\ & \le  C \| g-g(0)\|^2_{ \B(2,\infty)}\sum_{n=1}^\infty v_{2^{n+1}-1} 2^{3n} \left(\int_0^1 t^{2^n}|f(t)|\,dt\right)^2
\\ & \le  C \| g-g(0)\|^2_{ \B(2,\infty)}\sum_{m=0}^\infty
\sum_{j\in I(m)}j^2 \left(\int_0^1 t^j|f(t)|\,dt\right)^2v_{2j-1}
\\ & \le C \| g-g(0)\|^2_{ \B(2,\infty)}\| \hti(f)\|^2_{\Dv}.
\end{split}\end{equation*}
By Corollary~\ref{co:hilbertop}, the last quantity is less or equal
than\, $C \| g-g(0)\|^2_{ \B(2,\infty)}\| f\|^2_{\Dv},$ showing the
validity of \eqref{eqn11}.

Moreover, using Corollary~\ref{co:hilbertop} again, the remaining
terms in \eqref{eqn10} in\-vol\-ving $\hat{g}(1)$ and $\hat{g}(2)$
can easily be controlled by $\|g-g(0)\|^2_{ \B(2,\infty)}\|
f\|^2_{\Dv}$. This together with \eqref{eqn10} and \eqref{eqn11},
implies that $\mathcal{H}_g$ is bounded on $\Dv$ with $$\|
\hg\|\lesssim \| g-g(0)\|_{ \B(2,\infty)}.$$

Reciprocally, assume that $\mathcal{H}_g$ is bounded on $\Dv$. For
each $N \in \N$,
 denote $a_N=1-2^{-N}$ and consider the function $f_N$ defined, for $z \in \D$ as follows:
\begin{equation*}
f_N(z) =
(1-a_N)^{\frac{\lambda}{2}}\vg(a_N)^{-1/2}(1-a_Nz)^{\frac{1-\lambda}{2}},
\end{equation*}
By Lemma~\ref{Lemma:weights-in-D-hat}(ii)(vii), $\lambda>0$ can be
choosen  big enough so that
\begin{equation}\label{eqn15n}
\sup_N\| f_N\|_{\Dv}<\infty.
\end{equation}
We are going to see now that
\begin{equation*} \|\mathcal{H}_g(f_N) \|^2_{\Dv} \ge C 2^{-N} \sum_{j\in I(N)} \widetilde{g_{j+1}},\quad
N\in\N. \end{equation*}

By \eqref{eqn10}, the left hand side above is larger or equal than
\begin{equation*}\begin{split}
\\ &\sum_{n=0}^\infty 2^{2n}\left(\int_0^1
t^{2^{n+1}}f_N(t)\,dt\right)^2v_{2^{n+2}-1} \sum_{j\in I(n)}
\widetilde{g_{j+1}}
\\ & \ge 2^{2N}\left(\int_0^1 t^{2^{N+1}}f_N(t)\,dt\right)^2v_{2^{N+2}-1} \sum_{j\in I(N)} \widetilde{g_{j+1}}
\\ & \ge C2^{2N}\left(\int_{a_N}^1 t^{2^{N+1}}f_N(t)\,dt\right)^2\vg(a_N) \sum_{j\in I(N)} \widetilde{g_{j+1}}
\\ & \ge C2^{-N}\sum_{j\in I(N)} \widetilde{g_{j+1}}
\end{split}\end{equation*}
which together with \eqref{eqn15n} and the inequality $\int_0^1 f_0(t)\,dt\ge C>0$ gives that
$$ \| g-g(0)\|_{ \B(2,\infty)}\lesssim \| \hg\|. $$
\end{proof}
\par Now, we deal with the compactness of generalized Hilbert operators $\hg$.
\begin{proposition}\label{compactness}
Let $g\in H(\D)$ and $v\in \DD$ which satisfies the conditions \eqref{eq:minftyDv} and \eqref{eq:h1}.   Then
$\mathcal{H}_g$ is compact on $\Dv$ if and only $g\in b (2,\infty)$.
\end{proposition}

We will need the following lemma, which can be easily proved by
using \eqref{eq:vg2}, H\"{o}lder's inequality and \eqref{eq:minftyDv2}.

\begin{lemma}\label{le:c1}
Let $v$ be a radial weight such that
\eqref{eq:vg2} and \eqref{eq:minftyDv} are satisfied. Let $\{f_j\} _{j=1}^\infty $ be a
sequence in $\Dv$ such that
$\sup_{j}\|f_j\|_{\Dv}=K<\infty$ and $f_j\to 0$, as
$j\to\infty $, uniformly on compact subsets of\, $\D$. Then the
following assertions hold:
\begin{itemize}
\item[\rm(i)] $\lim_{j\to\infty}\int_0^1 |f_j(t)|\,dt=0$;
\item[\rm(ii)] $\mathcal H_g(f_j)\to 0$, as $j\to\infty$,
uniformly on compact subsets of $\D$ for each $g\in H(\D)$.
\end{itemize}
\end{lemma}

We also remind the reader that the norm convergence in $\Dv$,
$v\in\DD$, implies the uniform convergence on compact subsets of
$\D$ by \cite[Lemma~3.2]{PelSum14}. With these tools and from the
proof of Theorem~\ref{p=infty}, Proposition~\ref{compactness} can be
shown using standard techniques. Therefore, its proof will be
omitted. See \cite[Section $7$]{GaGiPeSis} or
\cite[Section~7]{PelRathg} for further details.

\subsection{Hilbert-Schmidt operators}

First, we observe that
\begin{equation}\label{eq:emdv}
\B(2,2)\subset \Dv \quad\text{if $v\in\DD$ satisfies that $\int_0^1
\Vg_2(s)\,ds<\infty$ }.
\end{equation}
In fact,
by Lemma~\ref{Lemma:weights-in-D-hat}(vi)
\begin{equation*}\begin{split}
 &\sup_{j\in\N}(j+1)v_{2j+1}\asymp \sup_{j\in\N}(j+1)\vg\left( 1-\frac{1}{2j+1}\right)
 \\ &\asymp\sup_{j\in\N}\int_{1-\frac{1}{j+1}}^{1-\frac{1}{j+2}}\Vg_2(s)\,ds
\le \int_0^1 \Vg_2(s)\,ds<\infty,
\end{split}\end{equation*}
which implies \eqref{eq:emdv}.

\begin{proposition}\label{p=2}
Let $g\in H(\D)$ and $v\in \DD$ which satisfies the conditions
\eqref{eq:minftyDv} and \eqref{eq:h1}. Then $\mathcal{H}_g \in
S_2(\Dv)$ if and only if $g\in \B(2,2)$. Moreover,
$$\| \hg\|_{ S_2(\Dv)}\asymp \| g-g(0)\|_{ \B(2,2)}.$$
\end{proposition}
\begin{proof}
Denote $e_0(z)=1$,
    $$
     e_n(z)=\frac{z^n}{\|z^n\|_{\Dv}}=\frac{z^n}{\sqrt{2n^2 v_{2n-1}}},\quad n\in\N \backslash \{0\},
    $$
and consider the basis $\left\{e_n\right\}_{n\in\N}$ of $\Dv$. If
$g(z)=\sum_0^{\infty} \go(k)z^k\in H(\D)$,
 since $\hg(e_0)(z)=\frac{g(z)-g(0)}{z}$, by \eqref{eq:emdv} $$\left\|\hg(e_0)\right\|^2_{\Dv}\asymp \|g-g(0)\|^2_\Dv\lesssim \| g-g(0)\|^2_{ \B(2,2)}.$$
\par On the other hand,
    \begin{equation}\begin{split}\label{eqp2}
    \left\|\hg(e_n)\right\|^2_{\Dv} &=\frac{|\go(1)|^2}{2(n+1)^2n^2 v_{2n-1}}
    \\ & +
\frac{1}{2n^2 v_{2n-1}}\sum_{k=1}^{\infty}
\frac{k^2\widetilde{g_{k+1}}}{(n+k+1)^2 }v_{2k-1},\quad n\in\N.
    \end{split}\end{equation}

Clearly,
    \begin{equation*}\begin{split}
    \sum_{n=1}^{\infty}
    \frac{1}{n^2(n+k+1)^2v_{2n-1}}
    &\gtrsim \frac{1}{v_{2k+1}}\sum_{m=1}^\infty \sum_{n=m(k+1)}^{(m+1)(k+1)} \frac{1}{n^2(n+k+1)^2}
    \\ &\asymp
    \frac{1}{(k+1)^3v_{2k+1}},\quad k\in\N.
    \end{split}\end{equation*}
The opposite inequality also holds, since $v\in\DD$,  from
Lemma~\ref{Lemma:weights-in-D-hat}(vi) and \eqref{eq:minftyDv}, we
get
    \begin{equation*}\begin{split}
    \sum_{n=1}^{k}
    \frac{1}{n^2(n+k+1)^2v_{2n-1}} &\asymp \frac{1}{(k+1)^2}\sum_{n=1}^{k}
    \frac{1}{\vg\left(1-\frac{1}{2n-1}\right)}\int_{1-\frac{1}{n}}^{1-\frac{1}{(n+1)}} \,ds
     \\ &\asymp \frac{1}{(k+1)^2}\int_0^{1-\frac{1}{(k+1)}}\frac{1}{\vg\left(s\right)} \,ds
     \\ & \lesssim  \frac{1}{(k+1)^2\int_{1-\frac{1}{(k+1)}}^{1-\frac{1}{(k+2)}}\Vg_2(s)\,ds}
\asymp  \frac{1}{(k+1)^3v_{2k}}
    ,\quad k\in\N.
    \end{split}\end{equation*}
For the rest of the values of $n$, using again
Lemma~\ref{Lemma:weights-in-D-hat}, and \eqref{eq:h1}
    \begin{equation*}
    \begin{split}
    \sum_{n=k+1}^{\infty}\frac{1}{n^2(n+k+1)^2v_{2n-1}} &\asymp \sum_{n=k+1}^{\infty}\frac{1}{n^2v_{2n-1}} \int_{1-\frac{1}{n}}^{1-\frac{1}{(n+1)}} \,ds \asymp \int_{1-\frac{1}{k+1}}^{1} \frac{1}{\Vg_2(s)}\,ds
    \\ & \lesssim \frac{1}{\int_0^{1-\frac{1}{k+1}} \Vg_4(s)\,ds} \lesssim \frac{1}{(k+1)^3v_{2k}}
    ,\quad k\in\N.
\end{split}\end{equation*}

From here, using Lemma~\ref{Lemma:weights-in-D-hat}(iv) and
\eqref{eqp2}, we deduce
    \begin{equation*}\begin{split}
    \sum_{n=1}^{\infty}\left\|\hg(e_n)\right\|^2_{\Dv} & =
    |\go(1)|^2\sum_{n=1}^{\infty}\frac{1}{2(n+1)^2n^2 v_{2n-1}}
\\ & +\sum_{k=1}^{\infty}
k^2\widetilde{g_{k+1}}
v_{2k-1}\sum_{n=1}^{\infty}\frac{1}{2n^2(n+k+1)^2 v_{2n-1}}
\\ & \asymp |\go(1)|^2+\sum_{k=1}^{\infty} (k+1) |\hat{g}(k+1)|^2\asymp
\|g-g(0)\|_{\B(2,2)},
   \end{split}\end{equation*}
proving our assertion.
\end{proof}

\subsection{Proof of Theorem~\ref{th:mainhg}. $\mathbf{(ii)\Rightarrow(i)}$. }
\par In order to obtain this implication in Theorem~\ref{th:mainhg} we use results from \cite{OFStudia94} on complex interpolation
 for the mixed norm spaces $\B(2,p)$ and from \cite[Theorem $2.6$]{Zhu} for Schatten classes.
Given $(X_0,X_1)$ a compatible pair of Banach spaces, we  denote by
$\left(X_0,X_1\right)_{[\theta]}$ the complex interpolating space of
exponent $\theta\in [0,1]$. With the notation from \cite{OFStudia94}
for $D$ and $A^{p,q}_{\delta,k}$, if one chooses the particular case
$D=\D$ in \cite{OFStudia94}, then $\B(2,q)=A^{2,q}_{1,1}$. In this
way, the following result on complex interpolation on the mixed norm
space $\B(2,q)$ is a consequence of \cite[Theorem~3.1]{OFStudia94}
(see also \cite[Theorem~B]{OFStudia94}).
\begin{lettertheorem}\label{th:OFinterpolation}
Let $0<q_0<q_1\le \infty$ and $\theta\in (0,1)$. If
$\frac{1}{q}=\frac{1-\theta}{q_0}+ \frac{\theta}{q_1},$
then
$$\left(\B(2,q_0), \B(2,q_1)\right)_{[\theta]}=\B(2,q).$$
\end{lettertheorem}

\begin{proposition}\label{p>2suf}
Let $2<p<\infty$
and $v\in \DD$ satisfying  \eqref{eq:minftyDv} and \eqref{eq:h1}.
If $g\in\B(2,p)$,
 then $\hg\in S_p(\Dv)$ and
$\| \hg\|_{ S_p(\Dv)}\lesssim \| g-g(0)\|_{ \B(2,p)}.$
\end{proposition}
\begin{proof}
Let us consider the linear operator $T(g)=\hg$. By
Proposition~\ref{p=2} the operator $T$ is bounded from $\B(2,2)$ to
$S_2(\Dv)$ with $$\| T(g)\|_{S_2(\Dv)}=\| \hg\|_{S_2(\Dv)}\asymp
\|g-g(0)\|_{\B(2,2)}.$$ Analogously, by Proposition~\ref{p=infty},
$T$ is bounded from $\B(2,\infty)$ to $S_\infty(\Dv)$ with $$\|
T(g)\|_{S_\infty(\Dv)}=\| \hg\|_{S_\infty(\Dv)}\asymp
\|g-g(0)\|_{\B(2,\infty)}.$$ So, the previous inequalities  together
with \cite[Theorem~$4.1.2$, p.~$88$]{BF76},
Theorem~\ref{th:OFinterpolation} and \cite[Theorem~$2.6$]{Zhu} imply
that  $T: \B(2,p)\to S_p(\Dv)$ is bounded.   This finishes the
proof.
\end{proof}

In order to deal with the case $0<p<2$, we will need two technical
lemmas.

\begin{lemma}\label{le:tech}
Let $v\in\DD$  satisfying \eqref{eq:h1}. Then, for any $q>0$ there is a constant $C(q,v)$ such that
$$\sum_{n=k}^\infty 2^{-3qn}v^{-q}_{2^{n+1}}\le C(q,v) 2^{-3qk}v^{-q}_{2^{k+1}}$$
for all $k\in\N\cup\{0\}.$
\end{lemma}
\begin{proof}
First, we prove that there exists $\gamma=\gamma(v)\in (0,1)$ such that
\begin{equation}\begin{split}\label{eq:tech1}
\int_{\frac{1+r}{2}}^1\frac{(1-s)^2}{\vg(s)}\,ds\le \gamma \int_{r}^{1}\frac{(1-s)^2}{\vg(s)}\,ds,\quad 0<r<1.
\end{split}\end{equation}
Taking into account that $v\in\DD$ and \eqref{eq:h1}
\begin{equation}\begin{split}\label{eq:tech3}
\int_{\frac{1+r}{2}}^1\frac{(1-s)^2}{\vg(s)}\,ds
& \le C(v) \left( \int_{0}^{\frac{1+r}{2}}\Vg_4(s)\,ds \right)^{-1}
\\ & \le C(v) \frac{(1-r)^3}{\vg\left(\frac{1+r}{2}\right)} \le C(v) \int_{r}^{\frac{1+r}{2}}\frac{(1-s)^2}{\vg(s)}\,ds,\quad
0<r<1,
\end{split}\end{equation}
which is equivalent to \eqref{eq:tech1}. Then,
Lemma~\ref{Lemma:weights-in-D-hat} and \eqref{eq:tech3} yield
\begin{equation*}\begin{split}
\sum_{n=k}^\infty  2^{-3qn}v^{-q}_{2^{n+1}}
 &\le C(q,v) \sum_{n=k}^\infty \left(\frac{1}{\vg\left(1-\frac{1}{2^{n+1}}\right)}\int_{1-\frac{1}{2^{n+1}}}^1(1-s)^2\,ds\right)^q
 \\ & \le C(q,v) \left(\int_{1-\frac{1}{2^{k+1}}}^1\frac{(1-s)^2}{\vg(s)}\,ds\right)^q  \sum_{n=k}^\infty
 \gamma^{n-k}.
\end{split}\end{equation*}
Since the last sum is convergent, all of the above is controlled by
\begin{equation*}C(q,v)
\left(\frac{1}{2^{3k}\vg\left(1-\frac{1}{2^{k+1}}\right)}\right)^q
\le C(q,v)\left(\frac{1}{2^{3k}v_{2^{k+1}}}\right)^q, \quad
k\in\N\cup\{0\}.
\end{equation*}
\end{proof}
\begin{lemma}\label{le:tech2}
Let $v\in\DD$  satisfying \eqref{eq:minftyDv}. Then, for any $q>0$ there is a constant $C(q,v)$ such that
$$\sum_{n=1}^k 2^{-qn}v^{-q}_{2^{n+1}}\le C(q,v) 2^{-qk}v^{-q}_{2^{k+1}}$$
for all $k\in\N.$
\end{lemma}
\begin{proof}
Now, we prove that there exists $\gamma=\gamma(v)\in (0,1)$ such
that
\begin{equation}\begin{split}\label{eq:tech5}
\int_{\frac{1+r}{2}}^1\Vg_2(s)\,ds \le \gamma \int_{r}^{1}\Vg_2(s)\,ds,\quad 0<r<1.
\end{split}\end{equation}
By \eqref{eq:minftyDv}
\begin{equation}\begin{split}\label{eq:tech4}
&\int_{\frac{1+r}{2}}^1\Vg_2(s)\,ds
\le \frac{C(v)}{ \int_{0}^{\frac{1+r}{2}}\frac{1}{\vg(s)}\,ds}
 \le \frac{C(v)}{ \int_{r}^{\frac{1+r}{2}}\frac{1}{\vg(s)}\,ds}
\\ & \le C(v)\frac{\vg\left(r\right)}{1-r}
\le C(v) \int_{r}^{\frac{1+r}{2}}\Vg_2(s)\,ds ,\quad 0<r<1,
\end{split}\end{equation}
which is equivalent to \eqref{eq:tech5}.
So, by \eqref{eq:tech4} and Lemma~\ref{Lemma:weights-in-D-hat}
\begin{equation*}\begin{split}
\sum_{n=1}^k 2^{-qn}v^{-q}_{2^{n+1}}
&\le C(q,v) \sum_{n=1}^k 2^{-qn}\vg\left(1-\frac{1}{2^{n+1}}\right)^{-q}
\\ &\le C(q,v) \sum_{n=1}^k \left( \int_{1-\frac{1}{2^{n+1}}}^1\Vg_2(s)\,ds \right)^{-q}
\\ &\le C(q,v) \left( \int_{1-\frac{1}{2^{k+1}}}^1\Vg_2(s)\,ds \right)^{-q} \sum_{n=1}^k\gamma^{k-n}
\\ &\le C(q,v) \left( \int_{1-\frac{1}{2^{k+1}}}^{1-\frac{1}{2^{k+2}}}\Vg_2(s)\,ds \right)^{-q}
\\ &\le C(q,v) 2^{-qk}\vg\left(1-\frac{1}{2^{k+1}}\right)^{-q} \le C(q,v) 2^{-qn}v^{-q}_{2^{k+1}},\quad k\in\N.
\end{split}\end{equation*}
\end{proof}

Now we are ready to prove the remaining case of (ii)$\Rightarrow$(i)
in Theorem~\ref{th:mainhg}.

\begin{proposition}\label{pr:sufp<2}
Let $v\in \DD$ satisfying \eqref{eq:minftyDv} and \eqref{eq:h1}. If
$0<p<2$ and $g\in\B(2,p)$, then $\hg\in \mathcal{S}_p(\Dv)$ and $\|
\hg\|_{ S_p(\Dv)}\lesssim \| g-g(0)\|_{ \B(2,p)}.$
\end{proposition}
\begin{proof}

Let $g(z)=\sum_{k=0}^\infty \widehat{g}(k)z^k \in\B(2,p)$.
We  use the  orthonormal basis $\{e_n\}_{n=0}^\infty$, where
\begin{equation}\label{eq:en}
 e_{0}(z)=1\quad\text{ and}\quad e_n(z)=\frac{\sum_{k+1\in I(n)}z^k}{\left(\sum_{k+1\in I(n)}k^2v_{2k-1}\right)^{1/2}},\quad n\in\N.
 \end{equation}
  Since $0<p\le 2$, $\B(2,p)\subset \B(2,2)$.
 Thus, by Proposition~\ref{p=2}, $\hg\in S_2(\Dv)$, and  in particular $\hg$ is a compact operator. Therefore, by
 \cite[Theorem~$1.26$]{Zhu} and
 \cite[Corollary~$1.32$]{Zhu} (applied to $\h_g^\star\h_g$)
 \begin{equation*}
 \|\h_g\|^p_{S^p(\Dv)}=\|\h_g^\star\h_g\|^{\frac{p}{2}}_{S^{\frac{p}{2}}(\Dv)}\le
 \sum_{n=0}^\infty\langle \hg e_n,\hg
 e_n\rangle_{\Dv}^{\frac{p}{2}}.
 \end{equation*}
Since $\hg(e_0)(z)=\frac{g(z)-g(0)}{z}$, by \eqref{eq:emdv}
$$\left\|\hg(e_0)\right\|^2_{\Dv}\asymp \|g-g(0)\|^2_\Dv\lesssim \|
g-g(0)\|^2_{ \B(2,p)}.$$ Moreover, for $n\in\N$
\begin{equation*}\begin{split}
&\hg(e_n)(z) =\sum_{j=0}^\infty(j+1)\hat{g}(j+1)\left(\int_0^1 t^je_n(t)\,dt\right)z^j
\\ & =\left(\sum_{k+1\in I(n)}(k+1)^2v_{2k+1}\right)^{-1/2}\sum_{j=0}^\infty(j+1)\hat{g}(j+1) \left(\sum_{m+1\in I(n)}\frac{1}{m+j+1} \right)z^j.
\end{split}\end{equation*}
So, it is enough to prove
 \begin{equation}\begin{split}\label{sp3}
\sum_{n=0}^\infty\langle \hg e_n,\hg e_n\rangle_{\Dv}^{\frac{p}{2}}\lesssim \|g-g(0)\|^p_{\B(2,p)}.
 \end{split}\end{equation}
By Lemma~\ref{Lemma:weights-in-D-hat}(v),  $\sum_{k+1\in
I(n)}(k+1)^2v_{2k+1}\asymp 2^{3n}v_{2^{n+1}}$, yielding
 \begin{equation}\begin{split}\label{sp2}
 &\sum_{n=1}^\infty\langle \hg e_n,\hg e_n\rangle_{\Dv}^{\frac{p}{2}}
 \lesssim \sum_{n=1}^\infty2^{-\frac{3pn}{2}}v_{2^{n+1}}^{-\frac{p}{2}}
 \left|\go(1)\sum_{m+1\in I(n)}\frac{1}{m+1} \right|^{p}
\\ & + \sum_{n=1}^\infty2^{-\frac{3pn}{2}}v_{2^{n+1}}^{-\frac{p}{2}}\left( \sum_{j=1}^\infty j^2\widetilde{g_{j+1}} \left(\sum_{m+1\in I(n)}\frac{1}{m+j+1} \right)^2v_{2j-1}
 \right)^{\frac{p}{2}}
 \\ & \le J_1+J_2+J_3,
 \end{split}\end{equation}
 where $J_1$ is the first sum on the right hand side, and the second
 sum is decomposed in $j \le 2^{n+1}-1$ ($J_2$) and $j \ge 2^{n+1}$
 ($J_3$).
By Lemma~\ref{le:tech},
\begin{equation}\begin{split}\label{sp2n}
 J_1 \lesssim |\go(1)|^p \sum_{n=1}^\infty2^{-\frac{3pn}{2}}v_{2^{n+1}}^{-\frac{p}{2}}\lesssim
\|g-g(0)\|^p_{\B(2,p)}.
 \end{split}\end{equation}
We now estimate for $J_2$, which satisfies
\begin{equation}\label{sp3n}
J_2 \asymp
\sum_{n=1}^\infty2^{-\frac{3pn}{2}}v_{2^{n+1}}^{-\frac{p}{2}}\left(
  \sum_{k=0}^n
  \sum_{j\in I(k)} j^2\widetilde{g_{j+1}} v_{2j-1}
 \right)^{\frac{p}{2}}.
\end{equation}
 Using
Lemma~\ref{Lemma:weights-in-D-hat} and Lemma~\ref{le:tech}, the
above can be bounded by
  \begin{equation}\begin{split}\label{sp4}
&\sum_{n=1}^\infty2^{-\frac{3pn}{2}}v_{2^{n+1}}^{-\frac{p}{2}}\left(
  \sum_{k=0}^n2^{2k}v_{2^{k+1}}
  \sum_{j\in I(k)} \widetilde{g_{j+1}}
 \right)^{\frac{p}{2}}
 \\ & \lesssim \sum_{n=1}^\infty2^{-\frac{3pn}{2}}v_{2^{n+1}}^{-\frac{p}{2}}
  \sum_{k=0}^n2^{pk}v^{p/2}_{2^{k+1}}
  \left(\sum_{j\in I(k)} \widetilde{g_{j+1}}
 \right)^{\frac{p}{2}}
 \\ & \lesssim \sum_{k=0}^\infty 2^{pk}v^{p/2}_{2^{k+1}}
  \left(\sum_{j\in I(k)} \widetilde{g_{j+1}}
 \right)^{\frac{p}{2}} \sum_{n=k}^\infty2^{-\frac{3pn}{2}}v_{2^{n+1}}^{-\frac{p}{2}}
 \\ & \lesssim \sum_{k=0}^\infty 2^{-\frac{pk}{2}}
  \left(\sum_{j\in I(k)} \widetilde{g_{j+1}}
 \right)^{\frac{p}{2}} \lesssim \|g-g(0)\|^p_{\B(2,p)}.
 \end{split}\end{equation}
To complete the estimation, we turn to $J_3$, which satisfies
\begin{equation*}
 J_3 \asymp \sum_{n=1}^\infty2^{-\frac{pn}{2}}v_{2^{n+1}}^{-\frac{p}{2}}\left(
  \sum_{k=n+1}^\infty 2^{-2k}
  \sum_{j\in I(k)} j^2\widetilde{g_{j+1}} v_{2j-1}
 \right)^{\frac{p}{2}}.
\end{equation*}

By Lemma~\ref{Lemma:weights-in-D-hat} and Lemma~\ref{le:tech2}, then
$J_3$ is controlled by
 \begin{equation*}\begin{split}
&\sum_{n=1}^\infty2^{-\frac{pn}{2}}v_{2^{n+1}}^{-\frac{p}{2}}\left(
  \sum_{k=n+1}^\infty v_{2^{k+1}}
  \sum_{j\in I(k)} \widetilde{g_{j+1}}
 \right)^{\frac{p}{2}}
 \\ & \lesssim \sum_{n=1}^\infty2^{-\frac{pn}{2}}v_{2^{n+1}}^{-\frac{p}{2}}
  \sum_{k=n+1}^\infty v^{p/2}_{2^{k+1}}
  \left(\sum_{j\in I(k)} \widetilde{g_{j+1}}
 \right)^{\frac{p}{2}}
 \\ & = \sum_{k=2}^\infty v^{p/2}_{2^{k+1}}
  \left(\sum_{j\in I(k)} \widetilde{g_{j+1}}
 \right)^{\frac{p}{2}} \sum_{n=1}^{k-1}2^{-\frac{pn}{2}}v_{2^{n+1}}^{-\frac{p}{2}}
 \\ & \lesssim \sum_{k=2}^\infty 2^{-\frac{kp}{2}}
  \left(\sum_{j\in I(k)} \widetilde{g_{j+1}}
 \right)^{\frac{p}{2}}\lesssim \|g-g(0)\|^p_{\B(2,p)}.
 \end{split}\end{equation*}
  This together with \eqref{sp2}, \eqref{sp2n}, \eqref{sp3n} and \eqref{sp4} gives \eqref{sp3}.
\end{proof}

\subsection{Proof of Theorem~\ref{th:mainhg}. $\mathbf{(i)\Rightarrow(ii)}$. }

In order to show the remaining part of the proof of
Theorem~\ref{th:mainhg} we will employ \cite[Theorem~$1.28$]{Zhu}.

\begin{proposition}\label{pr:necpge1}
Let $g\in H(\D)$ and $v\in \DD$ satisfying \eqref{eq:minftyDv} and \eqref{eq:h1}.
If $1\le p<\infty$ and $\hg\in \mathcal{S}_p(\Dv)$, then $g\in\B(2,p)$
and $$\|g-g(0)\|_{\B(2,p)}\lesssim \| \hg\|_{S_p(\Dv)}.$$
\end{proposition}
\begin{proof}
Let $\{e_n\}_{n=0}^\infty$ be the orthonormal basis defined in
\eqref{eq:en}. We write $M_0=|\hat{g}(1)|$ and
$$M_n=M_n(g,v)=\left(\sum_{k+1\in I(n)}k^2\widetilde{g_{k+1}} v_{2k-1}\right)^{1/2},\quad n\in\N.$$
Let us consider $N_g=\{n\in\N: M_n\neq 0 \}$ and
$$\sigma_n=M_n^{-1}\sum_{k+1\in I(n)}(k+1)\widehat{g}(k+1)z^k, \quad
n\in N_g.$$ The set $\{\sigma_n\}_{n\in N_g}$ is
   orthonormal.
Denote now $h_k := k^2 v_{2k-1} \widetilde{g_{k+1}}$. By
Lemma~\ref{Lemma:weights-in-D-hat},
 \begin{equation*}\begin{split}
&\left|\langle \hg e_n,\sigma_n \rangle_\Dv \right|
\\ & =\left(\sum_{k+1\in I(n)}k^2v_{2k-1}\right)^{-\frac12}
M_n^{-1}\sum_{k+1\in I(n)}h_k\left(\sum_{m+1 \in
I(n)}\frac{1}{m+k+1}\right)
\\ & \asymp \left(\sum_{k+1\in I(n)}k^2v_{2k-1}\right)^{-\frac12}\left(
\sum_{k+1\in I(n)}h_k \right)^{\frac12}
\\ & \asymp \left(2^{3n}v_{2^{n+1}}\right)^{-\frac12} \left(\sum_{k+1\in I(n)}h_k \right)^{\frac12} \asymp \left(2^{-n}\sum_{k+1\in
I(n)}\widetilde{g_{k+1}}\right)^{\frac12}, \quad \text{$n\in N_g$,
$n \ge 1$}.
\end{split}\end{equation*}
Hence, by \cite[Theorem~$1.28$]{Zhu}
\begin{equation*}\begin{split}
\infty & > \| \hg\|_{S_p(\Dv)}^p\ge\sum_{n\in N_g}\left|\langle \hg
e_n,\sigma_n\rangle_\Dv \right|^p \\ &\ge \sum_{n\in
N_g}\left(2^{-n}\sum_{k+1\in
I(n)}\widetilde{g_{k+1}}\right)^{\frac{p}{2}} \gtrsim \|
g-g(0)\|^p_{\B(2,p)}.
\end{split}\end{equation*}
\end{proof}

Finally, Theorem~\ref{th:mainhg} follows from
Propositions~\ref{p=infty}, \ref{p=2}, \ref{p>2suf}, \ref{pr:sufp<2}
and \ref{pr:necpge1}.

\subsection{Proof of  Corollary~\ref{co:hgbergman}.}
Since $\om\in\DD$, by \cite[Theorem~$4.2$]{PelRat} and
Lemma~\ref{Lemma:weights-in-D-hat},
\begin{equation}\begin{split}\label{eq:bergman}
\| f\|^2_{A^2_\om} &= |f(0)|^2\om(\D)+ \| f'\|^2_{A^2_\om*}
\\ & \asymp |f(0)|^2 + \int_\D |f'(z)|^2(1-|z|)\widehat{\om}(|z|)\,dA(z)
\\ &= \| f\|^2_{\Dv},\quad \text{ $f\in H(\D)$},
\end{split}\end{equation}
where $v(|z|)=(1-|z|)\om(|z|)$. Since $\om\in\DD$,  we have that
 $v\in\DD$ and $\vg(|z|)\asymp (1-|z|)^2\om(|z|)$,\,$z\in\D$. So,
using that $\widehat{\om}$ is a non-decreasing function
\begin{equation*}\begin{split}
& \sup_{0<r<1}\left(\int_r^1 \frac{\widehat{v}(s)}{(1-s)^2} \,ds\right)
\left(\int_0^r \frac{1}{\widehat{v}(s)} \,ds\right)
\\ & \asymp \sup_{0<r<1}\left(\int_r^1 \widehat{\om}(s) \,ds\right)
\left(\int_0^r \frac{1}{(1-s)^2\widehat{\om}(s)} \,ds\right)
<\infty.
\end{split}\end{equation*}
Moreover, by \eqref{M2condition}
\begin{equation*}
    \begin{split}
    &\sup_{0<
    r<1}\left(\int_{0}^r \frac{\widehat{v}(s)}{(1-s)^4}\,ds\right)^{\frac{1}{2}}
    \left(\int_{r}^1\frac{(1-s)^2}{\widehat{v}(s)}\,ds\right)^\frac{1}{2}
    \\ & \asymp \sup_{0<
    r<1}\left(\int_{0}^r \frac{\widehat{\om}(s)}{(1-s)^2}\,ds\right)^{\frac{1}{2}}
    \left(\int_{r}^1\frac{1}{\widehat{\om}(s)}\,ds\right)^\frac{1}{2}
    <\infty.
    \end{split}
    \end{equation*}
Therefore, $v$ satisfies both conditions,  \eqref{eq:minftyDv} and \eqref{eq:h1}. This together with \eqref{eq:bergman}
and Theorem~\ref{th:mainhg},  finishes the proof.


\end{document}